\documentclass[12pt]{article}
\usepackage[utf8]{inputenc}
\usepackage{amsmath}
\usepackage{amssymb}
\usepackage{amsthm}
\usepackage{amstext} 
\usepackage{etoolbox}
\usepackage{tikz}
\usepackage{mathrsfs}
\usepackage{mathtools}
\usepackage{bbm}
\usepackage{breqn}
\usepackage[draft=false]{hyperref}
\usepackage{cleveref}
\usepackage{authblk}
\usepackage{float}
\usepackage[style=numeric]{biblatex}
\usepackage[shortlabels]{enumitem}
\usepackage[margin=1.25in]{geometry}
\usepackage{array}   
\usepackage{booktabs}

\let\oldtocsection=\tocsection

\let\oldtocsubsection=\tocsubsection

\let\oldtocsubsubsection=\tocsubsubsection

\newcommand{\tocsection}[2]{\hspace{0em}\oldtocsection{#1}{#2}}
\newcommand{\tocsubsection}[2]{\hspace{2em}\oldtocsubsection{#1}{#2}}
\newcommand{\tocsubsubsection}[2]{\hspace{3em}\oldtocsubsubsection{#1}{#2}}

\newtheorem{thm}{Theorem}[section]
\newtheorem{cor}[thm]{Corollary}
\newtheorem{lem}[thm]{Lemma}
\newtheorem{prop}[thm]{Proposition}

\newtheorem{ass}[thm]{Assumption}
\newtheorem{rem}[thm]{Remark}
\newtheorem{defn}[thm]{Definition}
\newtheorem*{prop*}{Proposition}

\newtheorem{conj}[thm]{Conjecture}

\newcommand{\bigand}{\hspace{0.5cm}\text{and}\hspace{0.5cm}}

\newcommand{\modu}[1]{\left\lvert #1 \right\rvert}
\newcommand{\paren}[1]{\left( #1 \right)}
\newcommand{\brac}[1]{\left[ #1 \right]}
\newcommand{\curly}[1]{\left \{ #1 \right\}}

\newcommand{\dx}{\hspace{0.5mm}dx}

\newcommand{\N}{\mathbb{N}}

\newcommand{\R}{\mathbb{R}}

\newcommand{\E}{\mathbb{E}}

\renewcommand{\P}{\mathbb{P}}
\newcommand{\X}{\mathcal{X}}
\newcommand{\I}{\mathcal{I}}
\newcommand{\limn}{\lim_{n\to\infty}}
\renewcommand{\epsilon}{\varepsilon}

\newcommand\restr[2]{{
		\left.\kern-\nulldelimiterspace 
		#1 
		\vphantom{\big|} 
		\right|_{#2} 
}}

\newcolumntype{L}{>{$}l<{$}} 

\hypersetup{
	colorlinks = true,
	linkcolor = blue,
	urlcolor = cyan,
}
\title{Cutoff in the Bernoulli-Laplace Model With Unequal Colors and Urn Sizes}

\bibliography{refs.bib}
\author[1]{Thomas Griffin}
\author[2]{Bailey Hall}
\author[3]{Jackson Hebner}
\author[1]{David Herzog}
\author[4]{Denis Selyuzhitsky}
\author[5]{Kevin Wong}
\author[6]{John Wright}
\affil[1]{Iowa State University}
\affil[2]{Westmont College}
\affil[3]{University of Connecticut}
\affil[4]{The Michigan State University}
\affil[5]{University of California, Los Angeles}
\affil[6]{The Ohio State University}

\begin{document}

\maketitle

\begin{abstract}
We consider a generalization of the Bernoulli-Laplace model in which there are two urns and $n$ total balls, of which $r$ are red and $n - r$ white, and where the left urn holds $m$ balls. At each time increment, $k$ balls are chosen uniformly at random from each urn and then swapped. This system can be used to model phenomena such as gas particle interchange between containers or card shuffling. Under a reasonable set of assumptions, we bound the mixing time of the resulting Markov chain asymptotically in $n$ with cutoff at $\log{n}$ and constant window. Among other techniques, we employ the spectral analysis of \cite{K09} on the Markov transition kernel and the chain coupling tools of \cite{A22} and \cite{N19}. 
\end{abstract}

\tableofcontents

\section{Introduction}

In the classical Bernoulli-Laplace model, there are a total of $n$ balls split evenly between a left and a right urn. Of these balls, ${n \over 2}$ are red and ${n \over 2}$ are white. At each time increment, $k$ balls are selected uniformly at random without replacement from each urn and then swapped between the urns. 

We consider a generalization of this model where the left urn holds $m$ balls and the right urn $n-m$ balls, and where $r$ of the balls are red and $n-r$ of the balls are white. As in the classical case, $k$ balls are selected from both urns uniformly and without replacement to swap. Note that implicitly, $m,r,k$ are sequences in $n$ but we will suppress the reliance. We will let $X_t$ denote the number of red balls in the left urn after $t$ swaps. Transition probabilities for this process are given by 
\begin{equation}
    p(x, y) = \P(x - H_1^x + H_2^x = y)
\end{equation}
where $H_1^x$ and $H_2^x$ are independent hypergeometric random variables corresponding to the number of red balls moved out of the left urn and into the left urn during the swap, respectively. More precisely,
\begin{align}\label{clown}
    H_1^{x} \sim \mathrm{Hyp}(m, x, k) \bigand H_2^{x} \sim \mathrm{Hyp}(n - m, r - x, k).
\end{align}
This process is an irreducible, aperiodic Markov chain on the finite state space $$\mathcal X= \{\max(0,r+m-n), \max(0,r+m-n)+1,...,\min(m,r)\}$$ and so converges to a stationary distribution $\pi$. Given an initial distribution $\mu$ for $X_0$, we define the \textit{total variation distance} between the law of $X_t$ and the stationary distribution by $$\| \mu P^t - \pi \|_{TV} = \sup_{A \subset \X} \modu{\mu P^t(A) - \pi(A)},$$ where $P$ is the Markov transition kernel corresponding to the chain. We may then define the \textit{mixing time at $\varepsilon$} by $$t_{\mathrm{mix}}^{(n)}(\varepsilon) = \sup_{x \in \X} \inf \{t : \|\delta_x P^t - \pi \|_{TV} \leq \varepsilon\},$$ where $\delta_x$ is the point distribution at $x$ (meaning the chain starts at $X_0 = x \in \X$).

We say a sequence of Markov chains \textit{exhibits cutoff} in total variation if
$$
\limn \frac{t^{(n)}_{\mathrm{mix}}(\varepsilon)}{t^{(n)}_{\mathrm{mix}}(1-\varepsilon)} = 1 \hspace{4mm} \text{ for all fixed } \hspace{4mm} \varepsilon \in (0,1).
$$

Further, we say a sequence $a_n$ is a \textit{cutoff window} if there exists a constant $c(\varepsilon)$ for each $\epsilon \in (0,1)$ such that

$$
a_n = o\paren{t^{(n)}_{\mathrm{mix}}\paren{{1 \over 2}}} \bigand \modu{t^{(n)}_{\mathrm{mix}}(\varepsilon) - t^{(n)}_{\mathrm{mix}}(1-\varepsilon)} \le c(\varepsilon) a_n \hspace{4mm} \text{ for all } \hspace{4mm} n.
$$

\vspace{2mm}

This paper analyzes the mixing times of sequences $\{X_t^{(n)}\}_n$ of generalized Bernoulli-Laplace chains. To make this manageable, we introduce the following assumption.

\begin{ass}\label{banana}
    For each $n$, assume without loss of generality that $r,m \leq \frac{n}{2}$. Suppose $\limn {k \over n} = \gamma $, $\limn {m \over n } = h$ and $\limn {r \over n} = \eta$, and that
    \begin{align*}
    0 < \gamma \leq h \leq \frac{1}{2}, \hspace*{4mm}
    0 < \eta \leq {1 \over 2}, \hspace*{4mm}
    \gamma \neq h(1-h), \bigand
    \gamma \neq \frac{1}{2} \text{ if } h = {1 \over 2}.
    \end{align*}
\end{ass}

\begin{rem}
When proving asymptotic results, Assumption \ref{banana} allows us to make the following assumptions without loss of generality
\begin{align*}
    m, n-m \geq 2, \hspace*{4mm}
    k \neq \frac{m(n-m)}{n}, \bigand
    \text{if } k = m  \text{ then } m \neq \frac{n}{2}.
\end{align*}
Further, $$\X = \{0, \dots, \min(m,r)\}.$$
\end{rem}

For completeness, we explicitly provide the stationary distributions of the chains. The proof is simple and follows, for example, from the logic in \cite{T96}.

\begin{prop}
    The stationary distribution of $X_t^{(n)}$ is $\pi^{(n)} \sim \mathrm{Hyp}(n, r, m)$, with probability mass function $$\pi^{(n)}(j) = \frac{ {r\choose j} {n-r \choose m-j}}{{n \choose m}}.$$
\end{prop}

We now define for each $n$
\begin{equation}\label{orange}
    t_n := \frac{-\log{n}}{2\log{\modu{1 - \frac{kn}{m(n-m)}}}} = \frac{\log n}{2 \modu{\log \modu{1 - \frac{kn}{m(n-m)}}}}.
\end{equation}

\noindent \textbf{Notation}. Let $a_n$ and $b_n$ be sequences in $\R$. We write $a_n \lesssim b_n$ if there exists a constant $C \in \R$ such that $\modu{a_n} \leq C \modu{b_n}$ for all $n$ large enough. If $\limn \frac{a_n}{b_n} = 1$, we write $a_n \sim b_n$. To denote the probability of transitioning from a state $x$ into a set $S$ in $t$ steps, we write
\begin{align*}
    P_t(x,S) := \sum_{y \in S} p_t(x,y) = \P_x(X_t \in S),
\end{align*}
where $\P_x$ denotes the probability measure induced by the Markov chain starting from a state $X_0 = x \in \cal X$. We may also sometimes write $X_t^x$ to refer to a random variable induced by the same starting criterion, especially in the context of two chains started at different initial states. 

Our main result bounds $t_{\mathrm{mix}}^{(n)}(\varepsilon)$ within a constant distance of $t_n$ and is given below.

\begin{thm}
    \label{thm:main}
    Let $\{X_t^{(n)}\}_n$ be a sequence of generalized Bernoulli-Laplace chains satisfying Assumption \ref{banana}. Then there exist constants $N_1, N_2, c$ and $C$, all depending on $\epsilon, \gamma, \eta, h$, such that for all $n \geq N_1$, $$t_n - c \leq t_{\mathrm{mix}}^{(n)}(\epsilon),$$ and for all $n \geq N_2$, $$t_{\mathrm{mix}}^{(n)}(\epsilon) \leq t_n + C,$$
    with $t_n$ as in (\ref{orange}). 
\end{thm}

\begin{rem}
    By taking the maximums of a finite number of constants, Theorem \ref{thm:main} implies that for any sequence $\{X_t^{(n)}\}_n$ of Bernoulli-Laplace chains satisfying Assumption \ref{banana}, there exists a constant $C'$ such $$t_n - C' \leq t_{\mathrm{mix}}^{(n)}(\epsilon) \leq t_n + C'$$ for all $n$. This further implies that $t_{\mathrm{mix}}^{(n)}(\epsilon)$ exhibits cutoff at $t_n$ with constant window. 
\end{rem}

\begin{rem} If we assume further that $$\modu{\frac{1}{\log\modu{1 - \frac{\gamma}{\eta(1-\eta)}}} - \frac{1}{\log \modu{1 - \frac{kn}{m(n-m)}}}} = O \left( \frac{1}{\log n} \right),$$ then Theorem \ref{thm:main} implies (for possibly different $C'$) that
\begin{equation}
    \frac{-\log{n}}{2\log \modu{1 - \frac{\gamma}{\eta(1-\eta)}}} - C' \leq t_{\mathrm{mix}}^{(n)}(\epsilon) \leq \frac{-\log{n}}{2\log \modu{1 - \frac{\gamma}{\eta(1-\eta)}}} + C',
\end{equation} and so $t_{\mathrm{mix}}^{(n)}(\epsilon)$ exhibits cutoff at order $\log n$ with constant window.
\end{rem}

\begin{rem}
    Theorem \ref{thm:main} fails when $\gamma = h(1-h)$. In that case, $t_{\mathrm{mix}}^{(n)}(\epsilon)$ is usually a bounded sequence. For details, see Section \ref{sec:plantain}.
\end{rem}



\section{Lemmata} \label{sec:lemmata}

\subsection{Spectral Lemmata}

All of the eigenvalues and eigenfunctions for the Bernoulli-Laplace model are known (see \cite{K09}), but here we only need the first two. This is similar to the approach in \cite{A22}. 

\begin{lem}\label{eigdef}
    The Markov transition has first two right eigenfunctions
    \begin{align}
        s_1(x) &= 1 - \frac{n}{rm}x \\
        s_2(x) &= 1 - \frac{2(n-1)}{rm}x + \frac{(n-1)(n-2)}{r(r-1)m(m-1)}x(x-1)
    \end{align}
    with respective eigenvalues
    \begin{align}
        \lambda_1 &= 1 - \frac{nk}{m(n-m)} \\
        \lambda_2 &= 1 - \frac{2(n-1)k}{m(n-m)} + \frac{(n-1)(n-2)k(k-1)}{m(m-1)(n-m)(n-m-1)}.
    \end{align}
    That is,
    \begin{align*}
        \E_x[s_1(X_1)] = \lambda_1 s_1(x) \bigand \E_x[s_2(X_1)] = \lambda_2 s_2(x).
    \end{align*}
\end{lem}

\begin{proof}
    This can be checked by a direct calculation. 
\end{proof}

\begin{cor}
    For all $t \geq 0, x \in \mathcal X$, we have
    \begin{align*}
        \E_x[s_1(X_t)] &= \lambda_1^t s_1(x) \\
        \E_x[s_2(X_t)] &= \lambda_2^t s_2(x)
    \end{align*}
    by the Markov property.
\end{cor}
Doing some algebra, we find that
\begin{align*}
    s_1^2(x) = b_0 + b_1s_1(x) + b_2s_2(x),
\end{align*}
where
\begin{align*}
        b_0 = {(n-m)(n-r) \over (n-1)rm}, \hspace{0.5cm} b_1 = -{(n - 2r)(n - 2m) \over (n - 2)rm}, \hspace{0.5cm} b_2 = {n^2(r - 1)(m-1) \over (n - 1)(n - 2)rm}.
\end{align*}

\begin{cor}
    For any $x, t,$ we have
    $$\E_x[s_1^2(x)] = b_0 + b_1 \lambda_1^t s_1(x) + b_2 \lambda_2^t s_2(x).$$
\end{cor}

\begin{lem}\label{lem:Strawberry}
    Let $\lambda_1, \lambda_2$ be as above. Then $\lambda_1^2 - \lambda_2 = O({1 \over n}).$
\end{lem}

\begin{proof}
    Observe that
    \begin{align*}
        \lambda_1^2 - \lambda_2 &= \paren{1 - \frac{nk}{m(n-m)}}^2 - \paren{1 - \frac{2(n-1)k}{m(n-m)} + \frac{(n-1)(n-2)k(k-1)}{m(m-1)(n-m)(n-m-1)}} \\
        &= \frac{2k}{m(n-m)} + \paren{{n^2k^2 \over m^2 (n - m)^2}  - {(n - 2) (n - 1) k (k - 1) \over (n - m)(n - m - 1)m(m - 1)}}.
    \end{align*}
    The claim follows.
\end{proof}

\begin{lem} \label{lem:eigenvals2}
    Let $\lambda_2$ be as above, $t_n$ as in (\ref{orange}), and suppose $\gamma \neq h(1-h)$. Then $\lambda_2^{t_n} = O({1 \over n})$. 
\end{lem}

\begin{proof}
    We first show that $$\lambda_1^2 - \lambda_2 = \frac{2k}{m(n-m)} + \paren{{n^2k^2 \over m^2 (n - m)^2}  - {(n - 2) (n - 1) k (k - 1) \over (n - m)(n - m - 1)m(m - 1)}} \geq 0.$$
    Notice that for $n \geq 4$, we have $$\frac{4k^2}{m^2(n-m)^2} \leq \frac{2k}{m(n-m)}.$$ It thus suffices to show that
    \begin{align*}
        {(n^2+4)k^2 \over m^2(n-m)^2}  - {(n - 2) (n - 1) k (k - 1) \over (n-m)(n - m - 1)m(m - 1)} &\geq 0,
    \end{align*}
    which is equivalent to $$\frac{(n^2+4)k}{m(n-m)} - \frac{(n-2)(n-1)(k-1)}{(n-m-1)(m-1)} \geq 0.$$
    It is easy to check that the claim holds for $k=1$ and $k=m$. Since the left-hand side is linear in $k$, this implies that the claim holds for all $k \in \{1, \dots, m\}$.

    The condition $\gamma \neq h(1-h)$ implies $\limn \lambda_2 > 0$. We may thus assume without loss of generality that $\lambda_2 = \modu{\lambda_2}$. Therefore $\log{(\lambda_1^2)} - \log{\modu{\lambda_2}} =: a_n \geq 0$, and so
\begin{align*}
    \modu{\lambda_2^{t_n}} &= \exp\paren{\log{\modu{\lambda_2}}\frac{\log n}{|\log(\lambda_1^2)|}} \\
    &= \exp\paren{\left[\log\paren{\lambda_1^2} - a_n \right] \frac{\log n}{\modu{\log(\lambda_1^2)}}} \\
    &= \exp\paren{-\log n}\exp\paren{\frac{-a_n \log n}{\modu{\log(\lambda_1^2)}}} \\
    &\lesssim \frac{1}{n}.
\end{align*}
\end{proof}

\subsection{Hypergeometric Comparison Lemmata}


Denote the standard normal density by $\phi(x)$. For $l \in \X$, let
\begin{equation} \label{mango}
    p := \frac{l}{m}, \quad q := 1-p, \quad \sigma := 1 \vee \sqrt{kpq\paren{1-{k \over n}}}.
\end{equation} 

\begin{defn}
We say $Z$ has a discrete normal distribution on $\X$ with parameters $\zeta \in \R$ and $\xi > 0$ if
\begin{align*}
\P(Z=j) = {1 \over \xi \hspace*{.3mm} \mathcal{N}_{\zeta,\xi}} \phi\paren{{j-\zeta \over \xi}} \hspace*{1cm} \textrm{where} \hspace*{1cm} \mathcal{N}_{\zeta,\xi} := \sum_{x \in \X_k} \frac{1}{\xi} \phi\paren{x-\zeta \over \xi}
\end{align*}
for $j \in \mathcal{X}$, and denote $Z \sim \mathrm{dN}(\zeta, \xi)$.
\end{defn}
Further, let $x_j := {j - kp \over \sigma}$ and $\mathcal{N} := \mathcal{N}_{kp, \sigma}$.
Finally, we can give an important technical lemma, the proof of which follows Lemma 5.2 from \cite{A22}.

\begin{lem}\label{eggplant}
Given the above parameters, supposing Assumption \ref{banana} holds, letting $l = {rm \over n} + O(\sqrt{n})$, we have that 
\begin{align*}
    \mathcal{N} = 1 + O\paren{1 \over \sqrt{n}}.
\end{align*}
\end{lem}

\begin{proof}
Let $\mathcal{Z}$ be the standard normal distribution on $\R$. From the proof of Lemma 5.2 from \cite{A22}, we can obtain that 
$$
\mathcal{N} \ge 1 - \P \left(\mathcal{Z} \le {-kp \over \sigma } \right) - \P \left(\mathcal{Z} \ge {kq \over \sigma} \right) - {2 \over \sqrt{2\pi \sigma^2}}.
$$
Letting $l = {rm \over n} + O(\sqrt{n})$, we have
$$
p = {r \over n} + O \left( {1 \over \sqrt{n}}\right) 
\bigand
q = {r \over n} + O \left( {1 \over \sqrt{n}}\right) 
\bigand
\sigma \sim \sqrt{\gamma(1-\gamma) n \eta (1-\eta)}.
$$
Applying Chebyshev's inequality, we get
$$
\P\paren{\mathcal{Z} \le {-kp \over \sigma}} \le {\sigma^2 \over k^2 p^2} = O\left( {1 \over n} \right).
$$
Similarly, $\P\left( \mathcal{Z} = {kq \over \sigma} \right) = O\left( {1 \over n} \right)$. Of course, ${2 \over \sqrt{2 \pi \sigma^2}} = O\left( {1 \over \sqrt{n}} \right)$, and so the claim follows.
\end{proof}

\section{Bounds under Assumption \ref{banana}} \label{sec:lowerbound}

\subsection{Lower Bound}

In this section, we prove the lower bound of Theorem \ref{thm:main}, which we restate below.

\begin{prop} \label{prop:lower}
    Let $\{X_t^{(n)}\}_n$ be a sequence of generalized Bernoulli-Laplace chains satisfying Assumption \ref{banana}. Then there exist constants $N_1 := N_1(\epsilon, \gamma, \eta, h)$ and $c := c(\epsilon, \gamma, \eta, h)$ such that for all $n \geq N_1$, $$t_n - c \leq t_{\mathrm{mix}}^{(n)}(\epsilon).$$
\end{prop}
\begin{proof}
    We define 
\begin{align*}
    s(x) = \sqrt{n-1}s_1(x)
\end{align*}
and calculate the mean and variance of $s(X_t)$ both in the case where the chain is started from $0$ and with respect to the stationary distribution. We also let $t = t_n - c$, where $c$ is to be determined later. 

Clearly, $\E_\pi[s(X_t)] = 0$ and
\begin{align*}
    \text{Var}_\pi(s(X_t)) &= (n - 1){n^2 \over r^2 m^2} \text{Var}_\pi(X_t) = {(n - m)(n - r) \over rm} \le {(1 - h)(1 - \eta) \over \eta h} + 2
\end{align*}
for sufficiently large $n$. Now, 
\begin{align*}
    |\E_0[s(X_t)]| &= |\sqrt{n - 1} \lambda_1^t s_1(0)| \\
    &= \sqrt{n - 1}|\lambda_1|^{t_n}|\lambda_1|^{-c} \\
    &= \sqrt{n - 1 \over n} |\lambda_1|^{-c}.
\end{align*}
To calculate
\begin{align*}
    \text{Var}_0[s(X_t)] &= (n - 1)\E_0[s_{1}(X_t)^2] - \E_0[s(X_t)]^2,
\end{align*}
we first write 
\begin{align*}
    \E_0[s_1(X_t)^2] &= b_0 + b_1 \lambda_1^t + b_2 \lambda_2^t \\
    &= {(n - m)(n - r) \over (n - 1)rm } - {(n - 2m)(n - 2r) \over (n - 2)rm } \lambda_1^{t_n} \lambda_1^{-c} + {(m - 1)(r - 1) n^2 \over rm (n - 1) (n - 2)} \lambda_2^{t_n} \lambda_2^{-c}.
\end{align*}
By Lemma \ref{lem:eigenvals2}, $\lambda_2^{t_n} = O({1 \over n})$. Now consider these three terms in $(n - 1)\E_0[s_1(X_t)^2]$ in the limit as $n \to \infty$. The first, 
$$ {(n - 1) (n - m)(n - r) \over (n - 1)rm}\to {(1 - h)(1 - \eta) \over h\eta}.$$
The second vanishes, and the third is asymptotically bounded by $1$. It follows that
$$\text{Var}_0[s(X_t)] \leq {(1 - h)(1 - \eta) \over h \eta } + 2 =: K$$ for sufficiently large $n$.
We now define the sets $$ A_\alpha = \left\{x \in \mathcal X : \modu{s(x)} \leq \alpha \sqrt{K} \right\}$$ 
and
\begin{align*}
    B_{d, c} &= \left\{x : \min(|s(x) - \lambda_1^{-c}|, |s(x) + \lambda_1^{-c}|) \leq d\sqrt{K}\right\} \\
    &= \left\{x : |s(x) - \lambda_1^{-c}| \leq d\sqrt{K}\right\} \cup \left\{x : |s(x) + \lambda_1^{-c}| \leq d\sqrt{K}\right\}.
\end{align*}
The two sets reflect the two possibilities for the mean. Using Chebyshev's inequality, we calculate
\begin{align*}
    \pi_n(A_\alpha) \geq 1 - {1 \over \alpha^2}
\end{align*}
For $B_{d,c}$, we have
\begin{align*}
    P_t(0, B_{d, c}) &\geq P_t\paren{0, \curly{\modu{s(X_t) - \sqrt{n \over n - 1} \E_0[s(X_t)]} \leq d\sqrt{K}}} \\
    &= P_t\paren{0, \curly{\modu{s(X_t) -  \E_0[s(X_t)]} \leq d\sqrt{K}}} \\
    &\qquad + P_t\paren{0, \curly{\E_0[s(X_t)] - d\sqrt{K} \leq s(X_t) < \E_0[s(X_t)]\sqrt{n \over n - 1} - d\sqrt{K}}} \\
    &\qquad - P_t\paren{0, \curly{\E_0[s(X_t)] + d\sqrt{K} < s(X_t) \leq \E_0[s(X_t)]\sqrt{n \over n - 1} + d\sqrt{K}}} \\
    &\geq P_t\paren{0, \curly{\modu{s(X_t) -  \E_0[s(X_t)]} \leq d\sqrt{K}}} \\
    &\qquad - P_t\paren{0, \curly{\E_0[s(X_t)] + d\sqrt{K} < s(X_t) \leq \E_0[s(X_t)]\sqrt{n \over n - 1} + d\sqrt{K}}} \\
    &\geq 1 - {1 \over d^2} - {\modu{\paren{\sqrt{n \over n - 1} - 1} \E_0[s(X_t)]} \over {n\sqrt{n - 1} \over rm}} \\
    &\geq 1 - {1 \over d^2} - {K' \over \sqrt{n}}
\end{align*}
eventually, where $K'$ is a constant.
Notice that if 
\begin{equation} \label{eqn:disjointnesscondition}
    |\lambda_1|^{-c} - d\sqrt{K} \geq 2\alpha \sqrt{K}
\end{equation}
we have that $A_\alpha$ and $B_{d, c}$ are disjoint, implying that $P_t(0, A_\alpha) \leq {1 \over d^2} + {K' \over \sqrt{n}}$. Choosing $\alpha$ and $d$ such that ${1 \over \alpha^2}$ and ${1 \over d^2}$ are less than ${ \varepsilon \over 3}$, and then $c$ such that (\ref{eqn:disjointnesscondition}) is satisfied, we have 
\begin{align*}
    ||P_t(0, \cdot) - \pi_n||_{TV} &\geq |P_t(0, A_\alpha) - \pi_n(A_\alpha)| \\
    &\geq 1 - {1 \over \alpha^2} - {1 \over d^2} - {K' \over \sqrt{n}}\\
    &\geq 1 - \varepsilon
\end{align*} which gives the desired result. 

Observe that equality is (asymptotically) achieved in (\ref{eqn:disjointnesscondition}) when 
$$c(\varepsilon, \gamma, \eta, h) = {\log\paren{2\alpha +d} + \log \sqrt{K} \over \modu{\log\modu{1 - {\gamma \over h(1 - h)}}}}$$
which does not depend on $n$. 
\end{proof}

\subsection{Upper Bound} \label{sec:upperbound}

We now begin the proof of the upper bound in Theorem \ref{thm:main}.

\begin{defn}
    Define the following:
    \begin{align*}
    \mathcal{I}_{n}(\kappa) &= \left\{x \in \X : \left|x-\frac{rm}{n}\right| \leq \kappa \sqrt{n} \right\} \\
    F_n(\kappa) &= \left\{(x,y) \in \mathcal{I}_n(\kappa)^2 : \modu{x-y} \leq \frac{\sqrt{n}}{\kappa^3} \right\} \\
    \tau_{x,y}(\kappa) &= \min \{t : (Y_t^x,Y_t^y) \in F_n(\kappa)\}.
    \end{align*}
\end{defn}
\subsubsection{Path Coupling Contraction}

We use a generalization of the coupling introduced in \cite{N19}. Let $Y_{1,t}$ and $Y_{2,t}$ be generalized Bernoulli-Laplace chains with parameters $m, r, n$. We define a coupling between these chains as 
follows: In each chain, label the balls in the left urn with $\{1,\dots,m\}$ and the balls in the right urn with $\{m+1,\dots,n\}$ such that the red balls in any given urn have lower indices than the white balls in the same urn. Now, choose subsets $A \subset \{1,\dots,m\}$ and $B \subset \{m+1,\dots,n\}$ uniformly at random such that $\modu{A}=\modu{B}=k$. Then, in both chains, move the balls with labels in $A$ to the right urn and labels in $B$ to the left urn. Relabel the balls and repeat.

This is a coupling because each individual chain, $Y_{1,t}$ or $Y_{2,t}$, has the same behavior as $X_t$. However, there is a tendency for $Y_{1,t}$ and $Y_{2,t}$ to approach one another because the same indices are being swapped in each chain and indices are not independent of color. This is formalized below. 

\begin{lem}\label{lem:coupling}
Let $x,y \in \X$. Then for all $t \geq 0$, $$\E[|Y_t^x - Y_t^y|] \leq \left(1 - \frac{k(n-2k)}{m(n-m)}\right)^t |x-y|.$$ That is, each application of this path coupling's Markov transition is a strict contraction with coefficient $$1 - \frac{k(n-2k)}{m(n-m)} \in (0,1).$$
\end{lem}

\begin{proof}
    The proof is very similar to that of Lemma 3.4 in \cite{A22}. We may assume without loss of generality that $x > y$. By the Markov property, it suffices to prove for $t = 1$. Now consider the case when $x - y = 1$. Then $Y_1^x - Y_1^y$ can only takes values $-1,0, \text{ and } 1$. We directly calculate that $$\P(Y_1^x - Y_1^y = -1) = \frac{k}{m}\left(\frac{k}{n-m}\right)$$ and $$\P(Y_1^x - Y_1^y = 1) = \left( \frac{m-k}{m} \right) \left( \frac{n-m-k}{n-m}\right).$$ Thus
    \begin{align*}
        \E[|Y_1^x - Y_1^y|] &= \frac{k}{m}\left(\frac{k}{n-m}\right) + \left( \frac{m-k}{m} \right) \left( \frac{n-m-k}{n-m}\right) \\
        &= 1 - \frac{k(n-2k)}{m(n-m)}.
    \end{align*}
    When $x - y > 1$, we use the triangle inequality to get $$\E[|Y_1^x - Y_1^y|] \leq \sum_{i=0}^{x-y-1} \E[|Y_1^{y+i} - Y_1^{y+i+1}|] \leq (x-y)\left( 1 - \frac{k(n-2k)}{m(n-m)}\right).$$
\end{proof}

Now we state the following:

\begin{prop}\label{thenword}
    Suppose Assumption \ref{banana} is satisfied. Then for any $\kappa \in \N$ such that
    \begin{align} \label{eq:kappacon}
        \kappa^4\paren{1-{\gamma(1-2\gamma) \over h(1-h)}}^{\kappa} \leq {1 \over \kappa^2},
    \end{align}
    we have $$\P(\tau_{x,y}(\kappa) > t_n + \kappa) = O \left( \frac{1}{\kappa^2} \right).$$
\end{prop}

\begin{proof}

First observe that
\begin{align*}
\P(\tau_{x,y}(\kappa) > t_n + \kappa) \leq \P(X_{t_n}^x \not\in \I_{n}(\kappa)) + \P(X_{t_n}^y \not\in \I_{n}(\kappa))
+ \P(X_{t_n + \kappa}^x \not\in \I_{n}(\kappa)) 
\\
+ \P(X_{t_n + \kappa}^y \not\in \I_{n}(\kappa))
+ \P\paren{Y_{t_n}^x, Y_{t_n}^y \in \I_n(\kappa), \modu{Y_{t_n + \kappa}^x - Y_{t_n + \kappa}^y} > \frac{\sqrt{n}}{\kappa^3}},
\end{align*}
where we may use $X$ instead of $Y$ for the first four terms because $X$ and $Y$ share the same distributions by the coupling. We now work to asymptotically bound these terms by $\frac{1}{\kappa^2}$. For any constants $t \geq 0$ and $z \in \X$, we have
\begin{align*}
\P_z(X_t \not\in \I_n(\kappa)) &= \P_z \paren{\left|X_t - {rm \over n}\right| > \kappa \sqrt{n}}\\
&= \P_z \paren{s_1^2(X_t) > \kappa^2 {n^3 \over r^4}} \\ 
&\leq {1 \over \kappa^2}{r^4 \over n^3} \E_z[s_1^2(X_{t})] \\
&= {1 \over \kappa^2}{r^4 \over n^3} \paren{b_0 + b_1 \lambda_1^{t}s_1(z) + b_2 \lambda_2^{t}s_2(z)}.
\end{align*}
Two of the terms involve $t = t_n$ and the other two $t = t_n + \kappa$. Since $\modu{\lambda_1}, \modu{\lambda_2} \leq 1$, it suffices to bound for $t = t_n$. As $\frac{r^4}{n^3} = O(n)$, we need only show $b_0, b_1 \lambda_1^{t_n}s_1(z), b_2 \lambda_2^{t_n}s_2(z) = O({1 \over n})$. The first two follow easy because $b_0, b_1 = O({1 \over n })$ and $\lambda_1$ and $s_1(z)$ are bounded. For the final term, notice that $b_2$ and $s_2(z)$ is bounded. It therefore suffices to show that $\lambda_2^{t_n} = O({1 \over n})$, and this follows from Lemma \ref{lem:eigenvals2}. Thus $$\P_z(X_t \not\in \I_n(\kappa)) = O \left(\frac{1}{\kappa^2} \right).$$

We now bound the last term. Observe that
\begin{align*}
\P \paren{Y_{t_n}^x,Y_{t_n}^y \in \mathcal{I}_{n}(\kappa), \left|Y_{t_n + \kappa}^x - Y_{t_n + \kappa}^y\right| > {\sqrt{n} \over \kappa^3}} &\leq \max_{x,y \in \mathcal{I}_{n}(\kappa)} \P \paren{\left|Y_{\kappa}^x - Y_{\kappa}^y\right| \geq {\sqrt{n} \over \kappa^3}} \\
&\leq \max_{x,y \in \mathcal{I}_{n}(\kappa)} {\kappa^3 \over \sqrt{n}}\E \brac{|Y_{\kappa}^x - Y_{\kappa}^y|} \\
&\leq 2\kappa^4 \paren{1-{k(n-2k) \over m(n-m)}}^{\kappa} \\
&= O \left( {1 \over \kappa^2} \right).
\end{align*}
\end{proof}

\subsubsection{Mixing Times for Chains Started Close Together}
We now endeavor to prove the following: 
\begin{prop}\label{peach}
Suppose that Assumption \ref{banana} is satisfied. Then,
\begin{align*}
    \max_{(x,y) \in F_n(\kappa)} \| \delta_x P - \delta_y P\|_{TV} = O \left( {1 \over \kappa^2} \right).
\end{align*}
\end{prop}

\begin{proof}

Let $H_1^{x}$ and $H_2^{x}$ be hypergeometric random variables as in Equation (\ref{clown}). 
First, let $Z_a^b \sim \mathrm{dN}_k(kp_a^b,\sigma^b)$ for $a \in \{1,2\}$ and $b \in \{x,y\}$, where 
\begin{align*}
    p_1^x = {x \over m}, \hspace{0.3cm} p_2^x = q_1^x, \hspace{0.3cm} p_1^y = {y \over m}, \hspace{0.3cm} p_2^y = q_1^y, \bigand \sigma^b = 1 \vee \sqrt{kp_1^b p_2^b \paren{1- {k \over n}}}.
    \end{align*} 
Now observe that
\begin{align*}
    \mu_{H_2^x - H_1^x+ x}& - \mu_{H_2^y - H_1^y + y} \\
    &= (\mu_{H_2^x - H_1^x + x} - \mu_{H_2^x - Z_1^x + x}) + (\mu_{H_2^x - Z_1^x + x} - \mu_{Z_2^x - Z_1^x + x}) + (\mu_{Z_2^x - Z_1^x + x} - \mu_{Z_2^x - Z_1^y + x}) \\
    \quad
    &+ (\mu_{Z_2^x - Z_1^y + x} - \mu_{Z_2^y - Z_1^y + y}) + (\mu_{Z_2^y - Z_1^y + y} - \mu_{H_2^y - Z_1^y + y}) + (\mu_{H_2^y - Z_1^y + y} - \mu_{H_2^y - H_1^y + y}).
\end{align*}
By the independence of all the hypergeometrics and discrete normals above, we obtain 
\begin{align*}
\| \delta_x P - \delta_y P\|_{TV} &= \| \mu_{H_2^x - H_1^x+ x} - \mu_{H_2^y - H_1^y + y} \|_{TV} \\ 
&\leq \sum_{b \in \{x,y\}} \sum_{a=1}^2 \|\mu_{H_a^b} - \mu_{Z_a^b}\|_{TV} + \|\mu_{Z_2^x + x - y} - \mu_{Z_2^y}\|_{TV} + \|\mu_{Z_1^x} - \mu_{Z_1^y}\|_{TV}.
\end{align*}
We now bound the first term by proving the following local limit theorem.
\begin{prop}[Local limit theorem]
   Given Assumption \ref{banana} and using parameters in (\ref{mango}), if $l = {rm \over n} + O(\sqrt{n}),$ $H \sim \mathrm{Hyp}(m,l,k),$ and $Z \sim \mathrm{dN}_k(kp,\sigma)$, then $$\|\mu_H - \mu_Z\|_{TV} = O \bigg({1 \over \sqrt{n}} \bigg).$$
\end{prop}
\begin{proof}
Let $L := \sup\{ j \in \X : x_j \geq -\delta \sigma\}$ and $R := \inf \{ j \in \X : x_j \leq \delta \sigma \}$ for a fixed $\delta \in (0, {1 \over 2}]$. Observe that 
\begin{align*}
2\|\mu_H - \mu_Z \|_{TV} &= \sum_{j=L}^R |\P(H = j) - \P(Z = j)| \\
&\quad + \sum_{j=0}^{L-1} | \P(H = j) - \P(Z = j)| + \sum_{j=R+1}^k |\P(H = j) - \P(Z = j)| \\
&\qquad \qquad \qquad \qquad \qquad \qquad \qquad =: T_1 + T_2 + T_3 \hspace{2cm}
\end{align*}
First, by triangle inequality, 
$$
| \P(H = j) - \P(Z = j)| \le \P(H = j) + \P(Z = j).
$$
and by applying H{\"o}ffding's inequality from \cite{hoffdinginequality}, we obtain
$$
T_2 + T_3 \le {(k+1)\exp\paren{{-{1 \over 2}\delta^2 \sigma^2}} \over \mathcal{N} \sigma \sqrt{2 \pi }} + \P(H < L) + \P(H > R),
$$
where it is clear that 
\begin{align*}
    {(k+1)\exp \paren{{-{1 \over 2}\delta^2 \sigma^2}} \over \mathcal{N} \sigma \sqrt{2 \pi }} \to 0.
\end{align*} 
For the other two terms, recall that any hypergeometric distribution is the sum of $k$ independent Bernoulli random variables and $H$ has mean $kp$. Therefore, again applying H{\"o}ffding's inequality, 
\begin{align*}
\P(H < L) + \P(H > R) &= \P(H - kp < L - kp) + \P(H - kp > R - kp)  \\
&\le 2\exp\paren{-2{(L-kp)^2 \over k}} + 2\exp\paren{-2{(R-kp)^2 \over k}} \\
&\le 2\exp\paren{-2{\delta \sigma^2 \over k}} + 2\exp\paren{-2{\delta \sigma^2 \over k}} \\
&\to 0
\end{align*}
where the last step occurs since $\sigma^2 \sim \gamma(1-\gamma) n \eta(1-\eta)$ and $kp \sim \gamma n \eta$.
Finally, for $T_1$, observe that
\begin{align*}
    T_1 &\le \sum_{j=L}^R \left| {\phi(x_j) \over \sigma } - {\phi(x_j) \over \mathcal{N} \sigma } \right| + \sum_{j=L}^R \left| \P(H=j) - {\phi(x_j) \over \sigma } \right| \\
    &\le \sum_{j=L}^R {\phi(x_j) \over \sigma } \modu{1 - {1 \over \mathcal{N}} } + \sum_{j=L}^R \left| \P(H=j) - {\phi(x_j) \over \sigma } \right| \\
    &\le |\mathcal{N} - 1 | + \sum_{j=L}^R \left| \P(H=j) - {\phi(x_j) \over \sigma } \right|.
\end{align*}
The left term is $O\left( {1 \over \sqrt{n} }\right)$ by Lemma \ref{eggplant}. For the right term, notice that 
$$
\sum_{j=L}^R \left| \P(H=j) - {\phi(x_j) \over \sigma } \right| \le 2 \sum_{j=L}^{\lfloor kp \rfloor} \left| \P(H=j) - {\phi(x_j) \over \sigma } \right|
$$
From Theorem 1 of \cite{L07}, or in particular as Lemma 5.3 in \cite{A22}, we have that
$$
2 \sum_{j=L}^{\lfloor kp \rfloor} \left| \P(H=j) - {\phi(x_j) \over \sigma } \right| \leq {D \over \sigma} = O\left( {1 \over \sqrt{n}} \right)
$$
since $0 < {k \over n} < 1$ holds for all but finitely many $n$. This proves the result.
\end{proof} 

To complete the proof of Proposition \ref{peach}, we must show that $$\|\mu_{Z_2^x + x - y} - \mu_{Z_2^y}\|_{TV} + \|\mu_{Z_1^x} - \mu_{Z_1^y}\|_{TV} 
 = O \left( { 1 \over \kappa^2} \right).$$

We solve only $\|\mu_{Z_2^x + x - y} - \mu_{Z_2^y}\|_{TV} \lesssim {1 \over \kappa^2}$ since the proof for $\|\mu_{Z_1^x} - \mu_{Z_1^y}\|_{TV}$ follows the same method, but without a correction for $x-y$. First, without loss of generality assume $x \ge y$. 
Then, let $\mathcal{Y} := \X \cap \{\X + x-y \}$, $\mathcal{N}^x := \mathcal{N}_{kp_2^x,\sigma^x}$, $\mathcal{N}^y := \mathcal{N}_{kp_2^y, \sigma^y}$, \\ $\overline{y_j} := {j - kp_2^y \over \sigma^y} $
and $\overline{x_j} := {j - x + y - kp_2^x \over \sigma^x}$. Define $J_n := [{k \over 2} - 4 \kappa \sqrt{n}, {k \over 2} + 4 \kappa \sqrt{n} ]$, then
\begin{align*}
    2\|\mu_{Z_2^x + x - y} &- \mu_{Z_2^y}\|_{TV} \le \sum_{j \in \mathcal{Y}_n \cap J_n} \left| {\phi(\overline{x_j}) \over \mathcal{N}^x \sigma^x} - {\phi(\overline{y_j}) \over \mathcal{N}^y \sigma^y} \right| + \sum_{j \in \mathcal{Y}_n \cap J_n^c} {\phi(\overline{x_j}) \over \mathcal{N}^x \sigma^x} \\ 
    &+ \sum_{j \in \mathcal{Y}_n \cap J_n^c} {\phi(\overline{y_j}) \over \mathcal{N}^y \sigma^y} + \sum_{\X \setminus (\X + x -y)} \P(Z_2^y = j) + \sum_{(\X + x - y) \setminus \X} \P(Z_2^x = j - x + y) \\
    &\qquad \qquad \qquad \qquad \qquad \qquad \qquad \qquad  =: \mathcal{T}_1 + \mathcal{T}_2 + \mathcal{T}_3 + \mathcal{T}_4 + \mathcal{T}_5.
\end{align*}
We now work backwards. $\mathcal{T}_4$ and $\mathcal{T}_5$ are bounded analogously, as are $\mathcal{T}_2$ and $\mathcal{T}_3$.

For $\mathcal{T}_4$, observe that $x \ge y$, and by the space we are summing over, $j \le x-y-1$. Further, $\sigma^y \sim \sqrt{\gamma(1-\gamma) n \eta(1-\eta)}$. Using Lemma \ref{eggplant} and looking at large $n$, we obtain:
$$
\mathcal{T}_4 \leq \sum_{j=0}^{x-y-1} \P(Z_2^y = j) \leq (x-y)\P(Z_3 = x-y-1) = {x-y \over \mathcal{N}^y \sigma^y} \phi\paren{{x-y-1-kp_2^y \over \sigma^y}} \leq C_0 e^{-c_1 n}
$$
for constants $C_0,c_1 > 0$ independent of $n$.

For $\mathcal{T}_2$, let $\mathcal{Z}$ be the standard normal distribution over $\R$. Using the definition of $J_n$ with $\kappa \ge 0$, $n \ge 0$, $0 < \gamma < {1 \over 2}$, and $(x,y) \in F_n(\kappa)$, we have that asymptotically, 
$$
{k \over 2} + 4 \kappa \sqrt{n} \ge \kappa \sqrt{n} + x - y + kp_2^x + 1 \bigand {k \over 2} - 4\kappa \sqrt{n} \le -\kappa \sqrt{n} + x - y - 1 + kp_2^x.
$$ 
Therefore, we can use integral comparison, monotonicity of the standard normal density away from the mean, and Lemma \ref{eggplant} to obtain
$$
    \mathcal{T}_3 
    \le \sum_{j \in J_n^c} {\phi(\overline{x_j}) \over \mathcal{N}^x \sigma^x} 
    \le \int_{J_n^c} {\phi(\overline{x_j}) \over \mathcal{N}^x \sigma^x} \dx 
    \le {2 \over \mathcal{N}^x} \P\paren{\mathcal{Z} \ge {\kappa\sqrt{n} \over \sigma^x}} 
    \le {2 (\sigma^x)^2 \E[\mathcal{Z}^2] \over n \mathcal{N}^x} {1 \over \kappa^2} 
    = O \left( {1 \over \kappa^2} \right),
$$
where we use that $\sigma^x \sim \sqrt{n\eta(1-\eta)\gamma(1-\gamma)}$.

Now, for $\mathcal{T}_1$, we can use the result of $\left| {1 \over \sigma^x} - {1 \over \sigma^y} \right| = O({1 \over n})$ from \cite{A22}. Using Lemma \ref{eggplant}, observe: 
\begin{align*}
    \mathcal{T}_1 &= \sum_{j \in \mathcal{Y}_n \cap J_n} \left| {\phi(\overline{x_j}) \over \sigma^x}  - {\phi({\overline{y_j}}) \over \sigma^y} \right| + O \left({1 \over \sqrt{n}}\right) \\
    &\le \sum_{j \in \mathcal{Y}_n \cap J_n} \left| {\phi(\overline{x_j}) \over \sigma^x}  - {\phi({\overline{y_j}}) \over \sigma^x} \right| + \sum_{j \in \mathcal{Y}_n \cap J_n} \phi(\overline{y_j})\left| {1 \over \sigma^x} - {1 \over \sigma^y} \right| + O \left({1 \over \sqrt{n}} \right). \\
    &\qquad \qquad \qquad \qquad \qquad \qquad =: \mathscr{T}_1 + \mathscr{T}_2 + O \left({1 \over \sqrt{n}} \right).
\end{align*}
It is easy to see that, since $\phi$ is bounded above by $1$, 
$$
\mathscr{T}_2 \le 8 \kappa \sqrt{n} \cdot O \left({ 1 \over n} \right),
$$
which is also $O({1 \over \sqrt n})$. For $\mathscr{T}_1$, first observe that $\phi$ is Lipschitz, and let $K$ to be its Lipschitz constant. Then 
\begin{align*}
    \mathscr{T}_1 &\le {K \over \sigma^x} \sum_{j \in \mathcal{Y}_n \cap J_n} \left| \overline{y_j} - \overline{x_j} \right| \\
    &= {K \over \sigma^x} \bigg( {1 \over \sigma^x} \sum_{j \in \mathcal{Y}_n \cap J_n} \left| j - kp_2^y - (j - x + y - kp_2^x)\right| + \left| {1 \over \sigma^x} - {1 \over \sigma^y} \right| \sum_{j \in \mathcal{Y}_n \cap J_n} |j- kp_2^y| \bigg) .
\end{align*}
It can be seen that $\left| j - kp_2^y - (j - x + y - kp_2^x)\right|$ simplifies to $|x-y + kp_2^x - kp_2^y|$, and using the triangle inequality and the definitions of $p_2^b$, we obtain $|x-y|\left( 1 + {k \over m} \right)$, which is less than or equal to ${\sqrt{n} \over \kappa^3} \left( 1 + {k \over m} \right)$ by definition of $F_n(\kappa)$. 
For the right term, observe that $|j-kp_2^y| \le 8 \kappa \sqrt{n}$ by the definition of $J_n$, thus we now have $O\left( {1 \over n} \right) \cdot O(n)$, and now distributing ${K \over \sigma^x}$ and recognizing that $\sigma^x \sim O( \sqrt{n})$ concludes that the right term is also $O\left( {1 \over \sqrt{n}} \right)$. Thus we now have:
\begin{align*}
    \mathscr{T}_1 &\le {K \over (\sigma^x)^2 } \sum_{j \in \mathcal{Y}_n \cap J_n} {\sqrt{n} \over \kappa^3} \left( 1 + {k \over m} \right) \\
    &\le {K n \over (\sigma^x)^2}  \left( 1 + {k \over m} \right) \cdot { 1 \over \kappa^2 }\\ 
    &= O \left( {1 \over \kappa^2} \right)
\end{align*}
as desired.
\end{proof}

\subsubsection{Proof of Upper Bound}

We now prove the upper bound in Theorem \ref{thm:main} using the results derived in the previous two subsections.

\begin{prop} \label{prop:upper}
    Let $\{X_t^{(n)}\}_n$ be a sequence of generalized Bernoulli-Laplace chains satisfying Assumption \ref{banana}. Then there exist constants $N_2 := N_2(\epsilon, \gamma, \eta, h)$ and $C := (\epsilon, \gamma, \eta, h)$ such that for all $n \geq N_2$, $$t_{\mathrm{mix}}^{(n)}(\epsilon) \leq t_n + C.$$
\end{prop}

\begin{proof} Let $x,y \in \X$ and $A \subset \X$. Select $\kappa \in \N$ satisfying (\ref{eq:kappacon}). Define $t := t_n + \kappa + 1$. Then by Proposition \ref{thenword} and the strong Markov property,
\begin{align*}
    &|\P_x(X_t \in A) - \P_y(X_t \in A)|\\
    &\quad =|\P_x(Y_t \in A) - \P_y(Y_t \in A)| \\
    &\quad \leq 2\P(\tau_{x,y}(\kappa) > t-1) + | \P_x(Y_t \in A, \tau_{x,y}(\kappa) \leq t-1)- \P_y(Y_t \in A, \tau_{x,y}(\kappa) \leq t-1)| \\
    &\quad \le O \left( {1 \over \kappa^2} \right) + \max_{\substack{z,w \in F_n(\kappa) \\ s \in \{1, \dots, t\}}} |\P_z(Y_s \in A) - \P_w(Y_s \in A)|.
\end{align*}
Observe that for $s \in \{1,\dots, t\}$ and $z,w \in F_n(\kappa)$, 
$$ \modu{\P_z(Y_s \in A) - \P_w(Y_s \in A)} \leq \|\delta_z P - \delta_w P\|_{TV} = O \left( {1 \over \kappa^2} \right),$$
where we use Proposition \ref{peach}. 
Thus, 
\begin{equation}\label{eq:as}
    \max_{x \in \X} \|\delta_x P^t - \pi \|_{TV} \leq 2 \max_{x,y \in \X} \| \delta_x P^t - \delta_y P^t \|_{TV}
    = 2 \max_{x,y \in \X} |\P_x(X_t \in A) - \P_y(X_t \in A)| ~ = O \left( {1 \over \kappa^2} \right).
\end{equation}

Now let $C_3$ be the asymptotic bounding constant in (\ref{eq:as}). Then $$\max_{x \in \X} \|\delta_x P^t - \pi \|_{TV} \leq \frac{C_3}{\kappa^2}$$ for $n$ large enough. If it is not already true, increase $\kappa$ so that ${C_1 \over \kappa^2} \leq \epsilon$ in addition to (\ref{eq:kappacon}). Then $C = \kappa + 1$ works as the constant referenced in the theorem statement, and this completes the proof.
\end{proof}

\section{Bounds when $\gamma = h(1-h) > 0$}\label{sec:plantain}

We now investigate the critical case when $\gamma = h(1-h)$. We will ultimately show that the asymptotic behavior of $t_{\mathrm{mix}}^{(n)}(\epsilon)$ differs.

\begin{ass}\label{plantain}
    For each $n$, assume without loss of generality that $r,m \leq \frac{n}{2}$. Suppose $\limn {k \over n} = \gamma $, $\limn {m \over n } = h$ and $\limn {r \over n} = \eta$, and that $\gamma = h(1-h).$ Assume further that
    \begin{equation}\label{plantainrate}
        \modu{\lambda_1} = \modu{1 - \frac{nk}{m(n-m)}} \lesssim \frac{1}{\sqrt{n}}.
    \end{equation}
\end{ass}

\begin{thm}\label{thm:plantainbounds}
    Let $\{X_t^{(n)}\}_n$ be a sequence of Bernoulli-Laplace chains satisfying Assumption \ref{plantain}. Then there exists constants $N := N(\epsilon, \gamma, \eta, h)$ and $C := C(\epsilon, \gamma, \eta, h)$ such that for all $n \geq N$,
    \begin{equation}
        2 \leq t_{\mathrm{mix}}^{(n)}(\epsilon) \leq C.
    \end{equation}
\end{thm}

\begin{rem}
    Under Assumption \ref{plantain}, we no longer have cutoff at a multiple of $\log{n}$.
\end{rem}

We split the proof of Theorem \ref{thm:plantainbounds} across two subsections.

\subsection{The Lower Bound}

\begin{proof}
    We bound the total variation at time $t = 1$:
    \begin{align*}
        ||\delta_{0} P - \pi^{(n)}||_{TV} &\geq |\delta_0 P(\{k\}) - \pi^{(n)}(\{k\})| \\
        &= 1 - \pi^{(n)}(\{k\}) \\
        &= 1 - \frac{{r \choose k}{n-r \choose m-k}}{{n\choose m}}\\
        &= 1 - {r! m! (n-m)! (r-m)! \over k! n! (r-k)! (m-k)! (n-m-r+k)!}.
    \end{align*}
    Since $n,k,r,m$ are all of the same order, it is easy to see that the $n!$ term in the denominator dominates the fraction and send its limit to 0. Therefore, for any $\epsilon \in (0,1)$, we have for sufficiently large $n$ that $$||\delta_0 P - \pi^{(n)}||_{TV} \geq \epsilon.$$ This implies $t_{\mathrm{mix}}^{(n)}(\epsilon) \geq 2$.
\end{proof}

\subsection{The Upper Bound}

For this subsection, we introduce an alternative to $t_n$. Notice that the definition given in (\ref{orange}) is invalid under Assumption \ref{plantain}.

\begin{defn}
    When $\lambda_2 \neq 0$, let
    \begin{equation*}
        q_n = \frac{\log(n)}{|\log|\lambda_2||} = \frac{-\log(n)}{\log\modu{1 - \frac{2(n-1)}{m(n-m)}k + \frac{(n-1)(n-2)k(k-1)}{m(n-m)(m-1)(n-m-1)}}}.
    \end{equation*}
    When $\lambda_2 = 0$, let $q_n = 1$. 
\end{defn}

We will now prove an analogue of Proposition \ref{thenword} under Assumption \ref{plantain}.

\begin{prop}\label{thenword2}
    Suppose Assumption \ref{plantain} is satisfied. Then for any $\kappa \in \N$ such that
    \begin{align} \label{eq:kappacon2}
        \kappa^4\paren{1-{\gamma(1-2\gamma) \over h(1-h)}}^{\kappa} \leq {1 \over \kappa^2},
    \end{align}
    we have $$\P(\tau_{x,y}(\kappa) > q_n + \kappa) = O \left( \frac{1}{\kappa^2} \right).$$
\end{prop}

\begin{proof}
    The the proof is identical to that of Proposition \ref{thenword}, except that we bound
    $$\P_z(X_{q_n} \not\in \I_n(\kappa)) = {1 \over \kappa^2}{r^4 \over n^3} \paren{b_0 + b_1 \lambda_1^{q_n}s_1(z) + b_2 \lambda_2^{q_n}s_2(z)}.$$
    We need that $\lambda_2^{q_n} = O \left( {1 \over n} \right)$. If $\lambda_2 = 0$, then any asymptotic bound is trivially satisfied. If $\lambda_2 \neq 0$, then
    \begin{align*}
        \lambda_2^{q_n} &\leq |\lambda_2|^{q_n} \\
        &= \exp\paren{\log|\lambda_2| \frac{\log(n)}{|\log|\lambda_2||}} \\
        &= \frac{1}{n}.
    \end{align*}
    This proves the proposition.
\end{proof}

Following the argument of Section \ref{sec:upperbound}, we arrive at an analogue of Proposition \ref{prop:upper}.

\begin{prop} \label{prop:criticalupper}
    Let $\{X_t^{(n)}\}_n$ be a sequence of generalized Bernoulli-Laplace chains satisfying Assumption \ref{plantain}. Then there exist constants $N := N(\epsilon, \gamma, \eta, h)$ and $C := C(\epsilon, \gamma, \eta, h)$ such that for all $n \geq N$, $$t_{\mathrm{mix}}^{(n)}(\epsilon) \leq q_n + C.$$
\end{prop}

We now employ this proposition with the rate of convergence assumption (\ref{plantainrate}) to prove the upper bound of Theorem \ref{thm:plantainbounds}.

\begin{proof}
     It suffices to show that $q_n$ is bounded when Assumption \ref{plantain} holds. This is trivially satisfied for all $n$ such that $\lambda_2 = 0$, so assume $\modu{\lambda_2} \neq 0$. Using Lemma \ref{lem:Strawberry}, we see that
     \begin{align*}
         \lambda_2 &= \lambda_1^2 + O\left({1 \over n}\right) \\
         &= O\left({1 \over \sqrt{n}}\right)^2 + O\left({1 \over n}\right) \\
         &= O\left({1 \over n }\right).
     \end{align*}
     That is to say, there exists a $K$ such that $\modu{\lambda_2} \leq {K \over n}$. Thus $$\log\modu{\lambda_2} \leq \log(K) - \log(n).$$ Consider only $n$ large enough such that $\log(K) - \log(n) < 0.$ Then $$\modu{\log|\lambda_2|} \geq \log(n) - \log(K).$$ This implies that $$q_n = \frac{\log(n)}{\modu{\log|\lambda_2|}} \leq \frac{\log(n)}{\log(n) - \log(K)},$$ which is clearly bounded.
\end{proof}

\section{Conclusion} \label{sec:disc}

\subsection{Generalizations}
Theorem \ref{thm:main} generalizes pre-existing work on the two-color, two-urn Bernoulli-Laplace model. In particular, we extend the work from \cite{N19} on the lower bound and \cite{A22} on the upper bound to a model with uneven distributions of colors and urn sizes. We also show that mixing time is bounded under certain conditions in Theorem \ref{thm:plantainbounds}, which we believe to be a new result.

Another possible generalization comes from letting there be more than 2 colors. The complete spectrum of the Markov transition in this case is still known (see \cite{K09}), but the more complicated state space is harder to work with, and many proofs (e.g. that of Proposition \ref{peach}) breaks down.

Another case still unproven (though analogous to \cite{EN20}) is when $\gamma = 0$. We put forth the following conjecture, which is the analogue of the result in \cite{N19} and notably lacks constant window.
\begin{conj}
    Let $\{X_t^{(n)}\}_n$ be a sequence of Bernoulli-Laplace chains. For each $n$, assume without loss of generality that $r,m \leq \frac{n}{2}$. Suppose $\limn {k \over n} = 0$, $\limn {m \over n } = h$ and $\limn {r \over n} = \eta$. Then there exists constants $N,c,c',$ and $c''$ all depending on $\varepsilon,\eta,\gamma,$ and $h$ such that for all $n \geq N$,
    \begin{equation*}
        \frac{n \log n}{c(\varepsilon, \eta, \gamma, h) k} - \frac{k}{n}c'(\varepsilon, \eta, \gamma, h) \leq t_{mix}(\varepsilon) \leq t_n + \frac{c''(\varepsilon, \eta, \gamma, h) n}{k}\log \log n.
    \end{equation*}
\end{conj}

The last generalization we propose is to let $k$ and $r$ be random variables such that convergence ${k \over n} \to \gamma$ and ${r \over n} \to\eta$ in distribution. We are not aware of any previous literature on such a model.

\subsection{Numerical Results}

The asymptotic behavior proven in Theorems \ref{thm:main} and \ref{thm:plantainbounds} is readily observed in numeric data, which throughout this section is exact. The mixing times for two sequences of Bernoulli-Laplace chains satisfying Assumption \ref{banana} are plotted in Figures \ref{fig:balls1} and \ref{fig:balls2}. Notice that they resemble a constant multiple of $\log n$.

Figure \ref{fig:balls3} displays the mixing times for a sequence of Bernoulli-Laplace chains satisfying Assumption \ref{plantain}. Notice that they appear constant.

\begin{figure}[H]
    \centering
    \includegraphics[width=10cm]{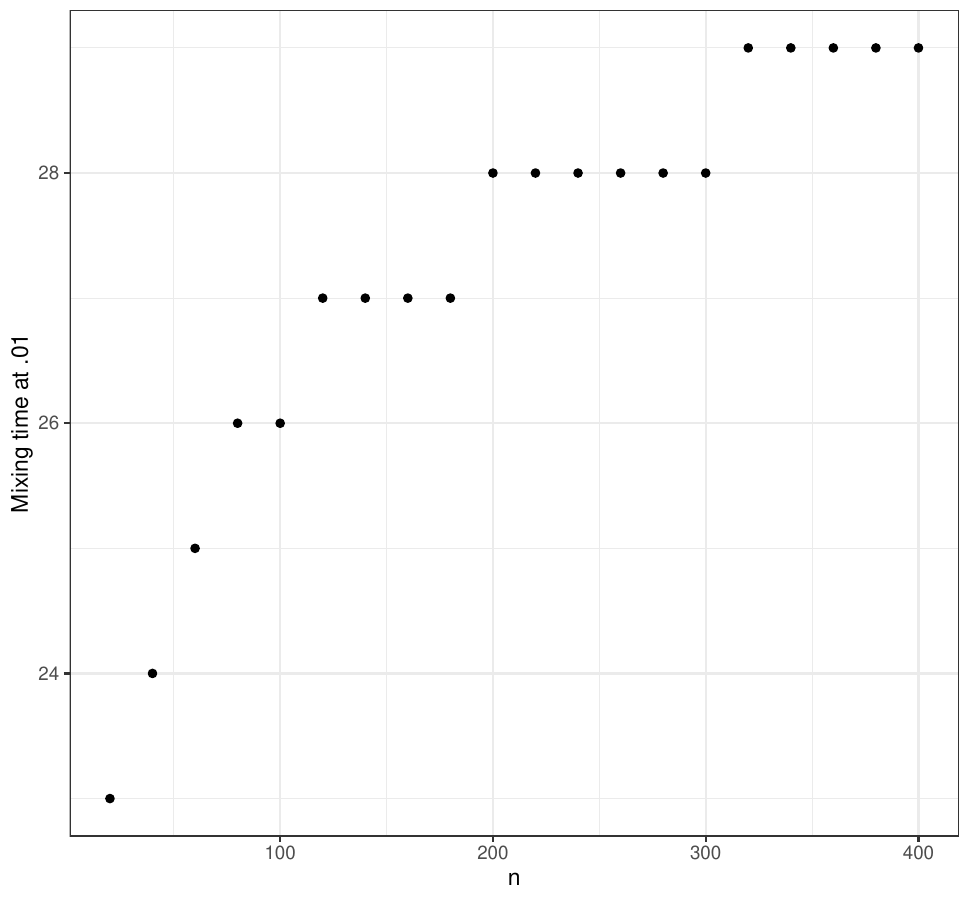}
    \caption{$t_{\mathrm{mix}}^{(n)}(0.01)$ when $\frac{k}{n} = 0.05, \frac{r}{n} = 0.40, m=r$}
    \label{fig:balls1}
\end{figure}

\begin{figure}[H]
    \centering
    \includegraphics[width=10cm]{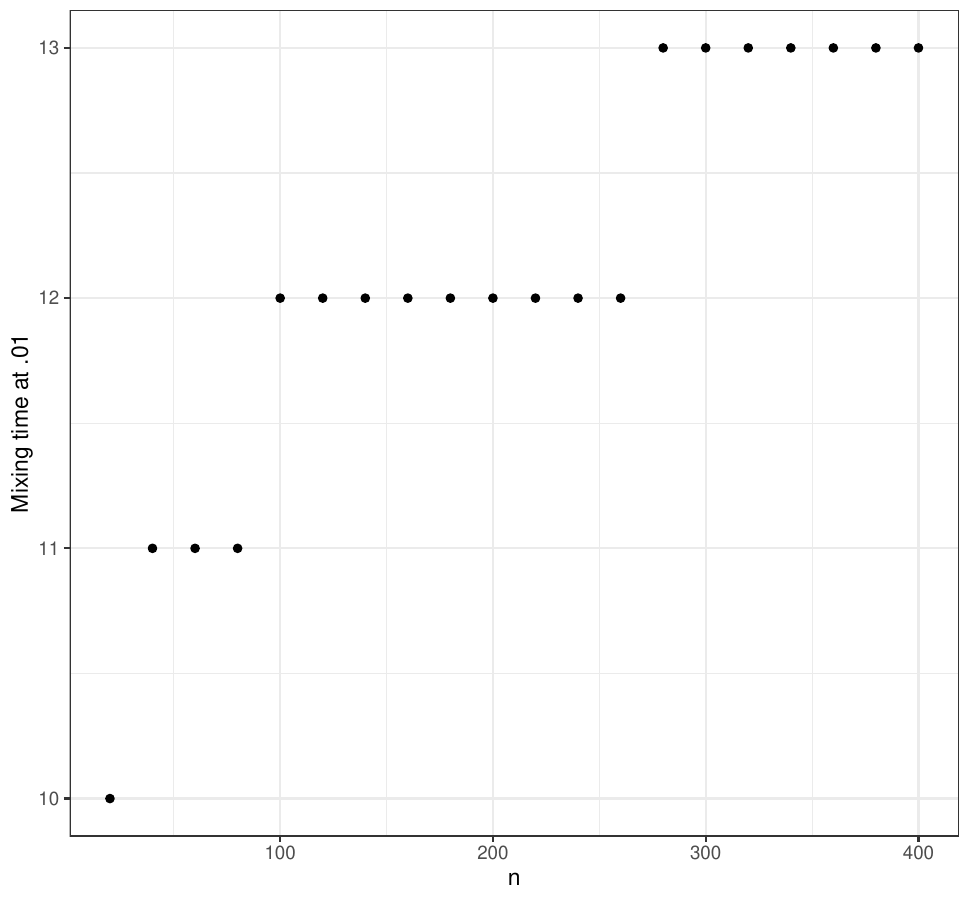}
    \caption{$t_{\mathrm{mix}}^{(n)}(0.01)$ when $\frac{k}{n} = 0.10, \frac{r}{n} = 0.40, m=r$}
    \label{fig:balls2}
\end{figure}

\begin{figure}[H]
    \centering
    \includegraphics[width=10cm]{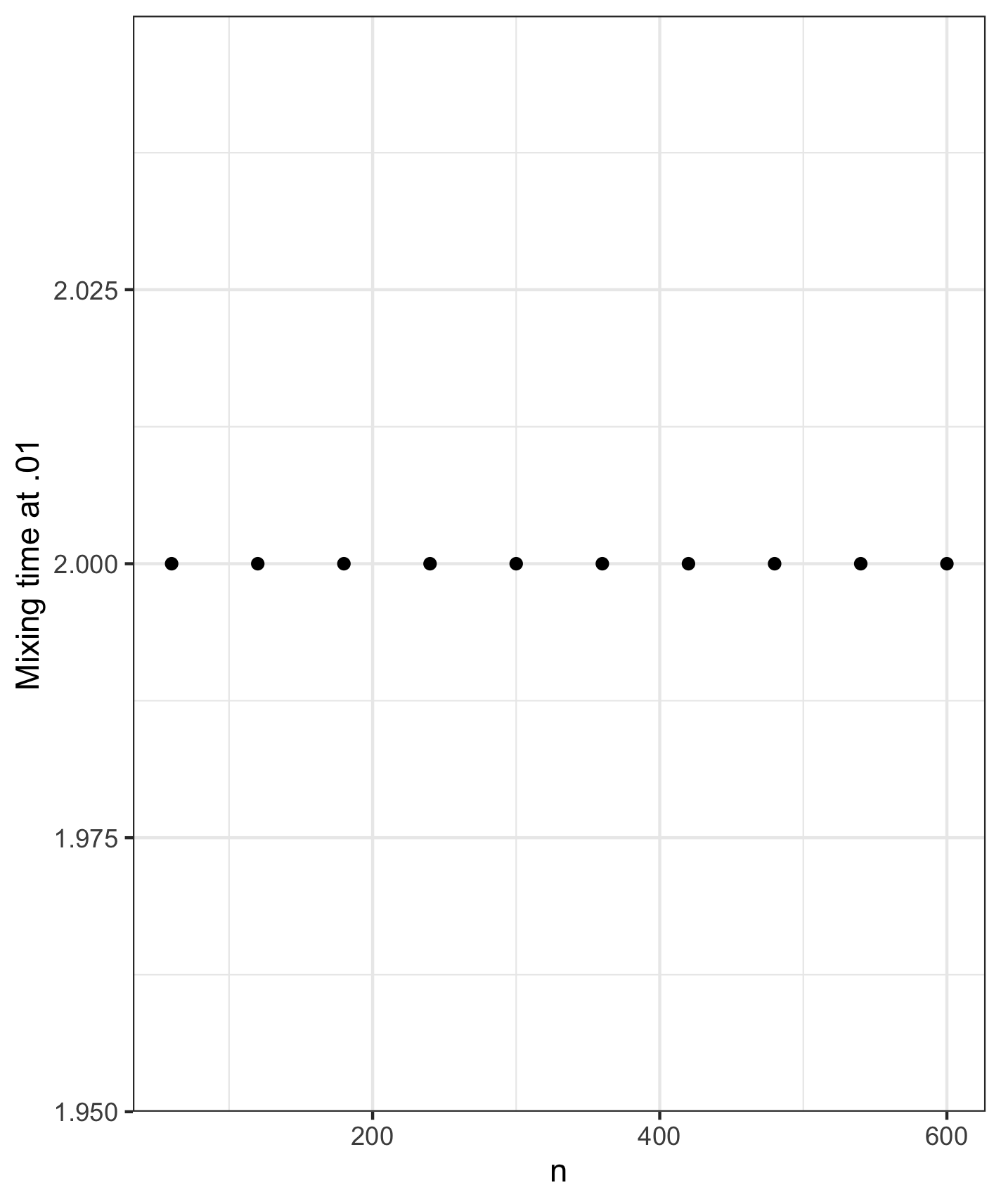}
    \caption{$t_{\mathrm{mix}}^{(n)}(0.01)$ when $\frac{k}{n} = 0.25, \frac{r}{n} = 0.50, m=r$}
    \label{fig:balls3}
\end{figure}

In Tables 1 and 2, we vary ${k \over n}$ and ${r \over n}$, respectively, while fixing $m = r$. In Table 3, we vary ${m \over n}$. 

Notice in Table 1 that the mixing times appear to reach a minimum at ${k \over n} = {r \over n}(1 - {r \over n})$. This makes sense in light of Theorems \ref{thm:main} and \ref{thm:plantainbounds}. For very low values of ${k \over n}$ (closer to $0$) and very high values of ${k \over n}$ (closer to ${r \over n}$), the mixing times quickly increase.

However, as shown in Tables 2 and 3, mixing times increase monotonically in ${m \over n}$. This makes sense given the behavior of $t_n$ under the same conditions. Tables 2 and 3 are so similar because the corresponding rows in each table differ only by their values of $\eta$, and as Theorem \ref{thm:main} proves, this allows for at most bounded difference in the mixing times. The data does consistently show, though, that mixing occurs faster $\eta = 0.50$ than $\eta < 0.50$.

\begin{table}\label{tab1}
\scalebox{0.75}{\hspace{-24mm}  
\begin{tabular}{|L|L|L|L|L|L|L|L|L|L|L|L|L|L|L|L|L|L|L|L|L|}
  \hline 
 $(k/n, n)$ & 50 & 100 & 150 & 200 & 250 & 300 & 350 & 400 & 450 & 500 & 550 & 600 & 650 & 700 & 750 & 800 & 850 & 900 & 950 & 1000 \\ 
  \hline 
  0.02 & 68 & 72 & 75 & 76 & 78 & 79 & 80 & 81 & 81 & 82 & 83 & 83 & 84 & 84 & 84 & 85 & 85 & 85 & 86 & 86 \\ 
  0.04 & 33 & 35 & 36 & 37 & 37 & 38 & 38 & 39 & 39 & 39 & 40 & 40 & 40 & 40 & 41 & 41 & 41 & 41 & 41 & 41 \\ 
  0.06 & 21 & 22 & 23 & 24 & 24 & 24 & 25 & 25 & 25 & 25 & 25 & 26 & 26 & 26 & 26 & 26 & 26 & 26 & 26 & 27 \\ 
  0.08 & 15 & 16 & 17 & 17 & 17 & 17 & 18 & 18 & 18 & 18 & 18 & 18 & 18 & 19 & 19 & 19 & 19 & 19 & 19 & 19 \\ 
  0.10 & 12 & 12 & 13 & 13 & 13 & 13 & 13 & 14 & 14 & 14 & 14 & 14 & 14 & 14 & 14 & 14 & 14 & 14 & 14 & 14 \\ 
  0.12 & 9 & 10 & 10 & 10 & 10 & 10 & 11 & 11 & 11 & 11 & 11 & 11 & 11 & 11 & 11 & 11 & 11 & 11 & 11 & 11 \\ 
  0.14 & 7 & 8 & 8 & 8 & 8 & 8 & 9 & 9 & 9 & 9 & 9 & 9 & 9 & 9 & 9 & 9 & 9 & 9 & 9 & 9 \\ 
  0.16 & 6 & 6 & 7 & 7 & 7 & 7 & 7 & 7 & 7 & 7 & 7 & 7 & 7 & 7 & 7 & 7 & 7 & 7 & 7 & 7 \\ 
  0.18 & 5 & 5 & 5 & 5 & 6 & 6 & 6 & 6 & 6 & 6 & 6 & 6 & 6 & 6 & 6 & 6 & 6 & 6 & 6 & 6 \\ 
  0.20 & 4 & 4 & 4 & 4 & 5 & 5 & 5 & 5 & 5 & 5 & 5 & 5 & 5 & 5 & 5 & 5 & 5 & 5 & 5 & 5 \\ 
  0.22 & 3 & 3 & 3 & 3 & 4 & 4 & 4 & 4 & 4 & 4 & 4 & 4 & 4 & 4 & 4 & 4 & 4 & 4 & 4 & 4 \\ 
  0.24 & 2 & 2 & 2 & 2 & 3 & 3 & 3 & 3 & 3 & 3 & 3 & 3 & 3 & 3 & 3 & 3 & 3 & 3 & 3 & 3 \\ 
  0.26 & 2 & 2 & 2 & 2 & 3 & 3 & 3 & 3 & 3 & 3 & 3 & 3 & 3 & 3 & 3 & 3 & 3 & 3 & 3 & 3 \\ 
  0.28 & 3 & 3 & 3 & 3 & 4 & 4 & 4 & 4 & 4 & 4 & 4 & 4 & 4 & 4 & 4 & 4 & 4 & 4 & 4 & 4 \\ 
  0.30 & 4 & 4 & 4 & 4 & 5 & 5 & 5 & 5 & 5 & 5 & 5 & 5 & 5 & 5 & 5 & 5 & 5 & 5 & 5 & 5 \\ 
  0.32 & 5 & 5 & 5 & 5 & 6 & 6 & 6 & 6 & 6 & 6 & 6 & 6 & 6 & 6 & 6 & 6 & 6 & 6 & 6 & 6 \\ 
  0.34 & 6 & 6 & 7 & 7 & 7 & 7 & 7 & 7 & 7 & 7 & 7 & 7 & 7 & 7 & 7 & 7 & 7 & 7 & 7 & 7 \\ 
  0.36 & 7 & 8 & 8 & 8 & 8 & 8 & 9 & 9 & 9 & 9 & 9 & 9 & 9 & 9 & 9 & 9 & 9 & 9 & 9 & 9 \\ 
  0.38 & 9 & 10 & 10 & 10 & 10 & 10 & 11 & 11 & 11 & 11 & 11 & 11 & 11 & 11 & 11 & 11 & 11 & 11 & 11 & 11 \\ 
  0.40 & 12 & 12 & 13 & 13 & 13 & 13 & 13 & 14 & 14 & 14 & 14 & 14 & 14 & 14 & 14 & 14 & 14 & 14 & 14 & 14 \\ 
  0.42 & 15 & 16 & 17 & 17 & 17 & 17 & 18 & 18 & 18 & 18 & 18 & 18 & 18 & 19 & 19 & 19 & 19 & 19 & 19 & 19 \\ 
  0.44 & 21 & 22 & 23 & 24 & 24 & 24 & 25 & 25 & 25 & 25 & 25 & 26 & 26 & 26 & 26 & 26 & 26 & 26 & 26 & 27 \\ 
  0.46 & 33 & 35 & 36 & 37 & 37 & 38 & 38 & 39 & 39 & 39 & 40 & 40 & 40 & 40 & 41 & 41 & 41 & 41 & 41 & 41 \\ 
  0.48 & 68 & 72 & 75 & 76 & 78 & 79 & 80 & 81 & 81 & 82 & 83 & 83 & 84 & 84 & 84 & 85 & 85 & 85 & 86 & 86 \\ 
   0.50 & +\infty & +\infty & +\infty & +\infty & +\infty & +\infty & +\infty & +\infty & +\infty & +\infty & +\infty & +\infty & +\infty & +\infty & +\infty & +\infty & +\infty & +\infty & +\infty & +\infty \\
  \hline
\end{tabular}}

\caption{$t_{\mathrm{mix}}^{(n)}(0.01)$ when $\frac{k}{n} \in \{0.02,0.04,\dots,0.48,0.50\}, \frac{r}{n} = 0.50$, $m=r$} 
\end{table}




\begin{table}
\scalebox{0.75}{\hspace{-12mm} \begin{tabular}{|L|L|L|L|L|L|L|L|L|L|L|L|L|L|L|L|L|L|L|L|L|}
  \hline
({r \over n},n) & 50 & 100 & 150 & 200 & 250 & 300 & 350 & 400 & 450 & 500 & 550 & 600 & 650 & 700 & 750 & 800 & 850 & 900 & 950 & 1000 \\ 
  \hline
0.02 & 2 & 2 & 2 & 2 & 2 & 2 & 2 & 2 & 2 & 2 & 2 & 2 & 2 & 2 & 2 & 2 & 2 & 2 & 2 & 2 \\ 
  0.04 & 8 & 8 & 9 & 9 & 9 & 10 & 10 & 10 & 10 & 10 & 10 & 10 & 10 & 10 & 10 & 10 & 10 & 10 & 10 & 10 \\ 
  0.06 & 13 & 14 & 15 & 15 & 15 & 15 & 16 & 16 & 16 & 16 & 16 & 16 & 16 & 16 & 16 & 17 & 17 & 17 & 17 & 17 \\ 
  0.08 & 19 & 20 & 20 & 21 & 21 & 21 & 21 & 22 & 22 & 22 & 22 & 22 & 22 & 22 & 23 & 23 & 23 & 23 & 23 & 23 \\ 
  0.10 & 23 & 24 & 25 & 26 & 26 & 26 & 27 & 27 & 27 & 27 & 28 & 28 & 28 & 28 & 28 & 28 & 29 & 29 & 29 & 29 \\ 
  0.12 & 28 & 29 & 30 & 31 & 31 & 32 & 32 & 32 & 33 & 33 & 33 & 33 & 33 & 34 & 34 & 34 & 34 & 34 & 34 & 35 \\ 
  0.14 & 31 & 33 & 34 & 35 & 36 & 36 & 37 & 37 & 38 & 38 & 38 & 38 & 39 & 39 & 39 & 39 & 39 & 40 & 40 & 40 \\ 
  0.16 & 36 & 38 & 39 & 40 & 41 & 41 & 41 & 42 & 42 & 43 & 43 & 43 & 43 & 44 & 44 & 44 & 44 & 44 & 45 & 45 \\ 
  0.18 & 39 & 42 & 43 & 44 & 45 & 45 & 46 & 46 & 47 & 47 & 47 & 48 & 48 & 48 & 49 & 49 & 49 & 49 & 49 & 50 \\ 
  0.20 & 42 & 45 & 47 & 48 & 49 & 49 & 50 & 50 & 51 & 51 & 52 & 52 & 52 & 53 & 53 & 53 & 53 & 54 & 54 & 54 \\ 
  0.22 & 46 & 49 & 50 & 52 & 52 & 53 & 54 & 54 & 55 & 55 & 56 & 56 & 56 & 57 & 57 & 57 & 57 & 58 & 58 & 58 \\ 
  0.24 & 49 & 52 & 54 & 55 & 56 & 57 & 57 & 58 & 58 & 59 & 59 & 60 & 60 & 60 & 61 & 61 & 61 & 62 & 62 & 62 \\ 
  0.26 & 52 & 55 & 57 & 58 & 59 & 60 & 61 & 61 & 62 & 62 & 63 & 63 & 64 & 64 & 64 & 65 & 65 & 65 & 65 & 66 \\ 
  0.28 & 54 & 58 & 60 & 61 & 62 & 63 & 64 & 64 & 65 & 66 & 66 & 66 & 67 & 67 & 67 & 68 & 68 & 68 & 69 & 69 \\ 
  0.30 & 57 & 60 & 62 & 64 & 65 & 66 & 67 & 67 & 68 & 68 & 69 & 69 & 70 & 70 & 70 & 71 & 71 & 71 & 72 & 72 \\ 
  0.32 & 59 & 63 & 65 & 66 & 67 & 68 & 69 & 70 & 70 & 71 & 71 & 72 & 72 & 73 & 73 & 73 & 74 & 74 & 74 & 75 \\ 
  0.34 & 61 & 65 & 67 & 68 & 69 & 71 & 71 & 72 & 73 & 73 & 74 & 74 & 75 & 75 & 75 & 76 & 76 & 76 & 77 & 77 \\ 
  0.36 & 63 & 66 & 69 & 70 & 72 & 72 & 73 & 74 & 75 & 75 & 76 & 76 & 77 & 77 & 78 & 78 & 78 & 79 & 79 & 79 \\ 
  0.38 & 64 & 68 & 70 & 72 & 73 & 74 & 75 & 76 & 76 & 77 & 78 & 78 & 79 & 79 & 79 & 80 & 80 & 80 & 81 & 81 \\ 
  0.40 & 65 & 69 & 72 & 73 & 75 & 76 & 76 & 77 & 78 & 79 & 79 & 80 & 80 & 80 & 81 & 81 & 82 & 82 & 82 & 83 \\ 
  0.42 & 66 & 70 & 73 & 74 & 76 & 77 & 78 & 79 & 79 & 80 & 80 & 81 & 81 & 82 & 82 & 83 & 83 & 83 & 84 & 84 \\ 
  0.44 & 67 & 71 & 74 & 75 & 77 & 78 & 79 & 79 & 80 & 81 & 81 & 82 & 82 & 83 & 83 & 84 & 84 & 84 & 85 & 85 \\ 
  0.46 & 68 & 72 & 74 & 76 & 77 & 78 & 79 & 80 & 81 & 81 & 82 & 82 & 83 & 83 & 84 & 84 & 85 & 85 & 85 & 86 \\ 
  0.48 & 68 & 72 & 75 & 76 & 78 & 79 & 80 & 80 & 81 & 82 & 82 & 83 & 83 & 84 & 84 & 85 & 85 & 85 & 86 & 86 \\ 
  0.50 & 68 & 72 & 75 & 76 & 78 & 79 & 80 & 81 & 81 & 82 & 83 & 83 & 84 & 84 & 84 & 85 & 85 & 85 & 86 & 86 \\ 
   \hline
\end{tabular}\label{tab2}}

\caption{$t_{\mathrm{mix}}^{(n)}(0.01)$ when $\frac{k}{n} = 0.02, \frac{r}{n} \in \{0.02,0.04,\dots,0.48,0.50\}$, $m=r$}

\end{table}

\begin{table}
\centering
\scalebox{0.75}{\hspace{-12mm} \begin{tabular}{|L|L|L|L|L|L|L|L|L|L|L|L|L|L|L|L|L|L|L|L|L|}
  \hline
 ({m \over n},n) & 50 & 100 & 150 & 200 & 250 & 300 & 350 & 400 & 450 & 500 & 550 & 600 & 650 & 700 & 750 & 800 & 850 & 900 & 950 & 1000 \\ 
  \hline
  0.02 & 2 & 2 & 2 & 2 & 2 & 2 & 2 & 2 & 2 & 2 & 2 & 2 & 2 & 2 & 2 & 2 & 2 & 2 & 2 & 2 \\ 
  0.04 & 6 & 6 & 7 & 7 & 7 & 7 & 7 & 7 & 7 & 8 & 8 & 8 & 8 & 8 & 8 & 8 & 8 & 8 & 8 & 8 \\ 
  0.06 & 10 & 11 & 12 & 12 & 12 & 12 & 12 & 13 & 13 & 13 & 13 & 13 & 13 & 13 & 13 & 13 & 13 & 14 & 14 & 14 \\ 
  0.08 & 14 & 15 & 16 & 17 & 17 & 17 & 17 & 18 & 18 & 18 & 18 & 18 & 18 & 19 & 19 & 19 & 19 & 19 & 19 & 19 \\ 
  0.10 & 19 & 20 & 21 & 21 & 22 & 22 & 22 & 23 & 23 & 23 & 23 & 24 & 24 & 24 & 24 & 24 & 24 & 24 & 24 & 25 \\ 
  0.12 & 22 & 24 & 25 & 26 & 26 & 27 & 27 & 28 & 28 & 28 & 28 & 29 & 29 & 29 & 29 & 29 & 29 & 29 & 30 & 30 \\ 
  0.14 & 27 & 28 & 30 & 30 & 31 & 31 & 32 & 32 & 33 & 33 & 33 & 33 & 34 & 34 & 34 & 34 & 34 & 35 & 35 & 35 \\ 
  0.16 & 30 & 32 & 34 & 35 & 35 & 36 & 36 & 37 & 37 & 37 & 38 & 38 & 38 & 39 & 39 & 39 & 39 & 39 & 39 & 40 \\ 
  0.18 & 34 & 36 & 38 & 39 & 40 & 40 & 41 & 41 & 42 & 42 & 42 & 43 & 43 & 43 & 43 & 44 & 44 & 44 & 44 & 44 \\ 
  0.20 & 37 & 40 & 42 & 43 & 44 & 44 & 45 & 45 & 46 & 46 & 47 & 47 & 47 & 47 & 48 & 48 & 48 & 48 & 49 & 49 \\ 
  0.22 & 41 & 44 & 45 & 46 & 47 & 48 & 49 & 49 & 50 & 50 & 51 & 51 & 51 & 52 & 52 & 52 & 52 & 53 & 53 & 53 \\ 
  0.24 & 44 & 47 & 49 & 50 & 51 & 52 & 52 & 53 & 54 & 54 & 54 & 55 & 55 & 55 & 56 & 56 & 56 & 57 & 57 & 57 \\ 
  0.26 & 47 & 50 & 52 & 53 & 54 & 55 & 56 & 57 & 57 & 58 & 58 & 58 & 59 & 59 & 59 & 60 & 60 & 60 & 61 & 61 \\ 
  0.28 & 50 & 53 & 55 & 57 & 58 & 59 & 59 & 60 & 60 & 61 & 61 & 62 & 62 & 63 & 63 & 63 & 64 & 64 & 64 & 64 \\ 
  0.30 & 53 & 56 & 58 & 59 & 61 & 62 & 62 & 63 & 64 & 64 & 65 & 65 & 65 & 66 & 66 & 66 & 67 & 67 & 67 & 68 \\ 
  0.32 & 55 & 59 & 61 & 62 & 63 & 64 & 65 & 66 & 66 & 67 & 68 & 68 & 68 & 69 & 69 & 69 & 70 & 70 & 70 & 71 \\ 
  0.34 & 57 & 61 & 63 & 65 & 66 & 67 & 68 & 68 & 69 & 70 & 70 & 71 & 71 & 72 & 72 & 72 & 73 & 73 & 73 & 73 \\ 
  0.36 & 59 & 63 & 65 & 67 & 68 & 69 & 70 & 71 & 72 & 72 & 73 & 73 & 74 & 74 & 74 & 75 & 75 & 75 & 76 & 76 \\ 
  0.38 & 61 & 65 & 68 & 69 & 70 & 71 & 72 & 73 & 74 & 74 & 75 & 75 & 76 & 76 & 77 & 77 & 77 & 78 & 78 & 78 \\ 
  0.40 & 63 & 67 & 69 & 71 & 72 & 73 & 74 & 75 & 76 & 76 & 77 & 77 & 78 & 78 & 79 & 79 & 79 & 80 & 80 & 80 \\ 
  0.42 & 65 & 68 & 71 & 72 & 74 & 75 & 76 & 77 & 77 & 78 & 78 & 79 & 79 & 80 & 80 & 81 & 81 & 81 & 82 & 82 \\ 
  0.44 & 65 & 70 & 72 & 74 & 75 & 76 & 77 & 78 & 79 & 79 & 80 & 80 & 81 & 81 & 82 & 82 & 82 & 83 & 83 & 83 \\ 
  0.46 & 67 & 71 & 73 & 75 & 76 & 77 & 78 & 79 & 80 & 80 & 81 & 82 & 82 & 82 & 83 & 83 & 84 & 84 & 84 & 85 \\ 
  0.48 & 67 & 72 & 74 & 76 & 77 & 78 & 79 & 80 & 81 & 81 & 82 & 82 & 83 & 83 & 84 & 84 & 85 & 85 & 85 & 85 \\ 
  0.50 & 68 & 72 & 75 & 76 & 78 & 79 & 80 & 81 & 81 & 82 & 83 & 83 & 84 & 84 & 84 & 85 & 85 & 85 & 86 & 86 \\ 
   \hline
\end{tabular}\label{tab3}}

\caption{$t_{\mathrm{mix}}^{(n)}(0.01)$ when $\frac{k}{n} = 0.02, {r \over n} = 0.50, \frac{m}{n} \in \{0.02,0.04,\dots,0.48,0.50\}$}

\end{table}

\section{Acknowledgements}
All authors were supported by NSF grant 1950583 for work done during the 2023 Iowa State Mathematics REU.

\newpage
\printbibliography
\end{document}